\documentclass[11pt]{article}

\usepackage{latexsym,amsmath,amsthm,amssymb,amscd}

\newfont{\bb}{msbm10 at 11pt}

\def\r{\hbox{\bb R}}

\def\c{\hbox{\bb C}}
\def\s{\hbox{\bb S}}

\def\amb{\mathcal{N}}

\def\campo{\mathfrak{X}}

\newcommand{\derv}[2]{\dfrac{d #1}{d #2}}
\newcommand{\parc}[2]{\dfrac{\partial #1}{\partial #2}}

\newcommand{\abs}[1]{\left\vert #1 \right\vert}
\newcommand{\set}[1]{\left\{#1\right\}}
\newcommand{\metag}[2]{g\left( #1 , #2 \right) }
\newcommand{\metal}[2]{\left( #1 , #2 \right) }
\newcommand{\meta}[2]{\langle #1 , #2 \rangle }
\newcommand{\eps}{\varepsilon}

\newcommand{\camb}{\overline{\nabla}}
\newcommand{\ricd}{{\rm Ric _\phi ^\infty }}
\newcommand{\scad}{ R _\phi ^\infty }
\newcommand{\derb}[3]{\left.\frac{d #1}{d #2}\right\vert _{#3} }

\numberwithin{equation}{section}

\begin{document}


\theoremstyle{plain}\newtheorem{lemma}{Lemma}[section]
\theoremstyle{plain}\newtheorem{definition}{Definition}[section]
\theoremstyle{plain}\newtheorem{theorem}{Theorem}[section]
\theoremstyle{plain}\newtheorem{proposition}{Proposition}[section]
\theoremstyle{plain}\newtheorem{remark}{Remark}[section]
\theoremstyle{plain}\newtheorem{corollary}{Corollary}[section]

\begin{center}
\rule{15cm}{1.5pt} \vspace{.6cm}

{\Large \bf Gradient Schr\"{o}dinger Operators, \\[3mm] Manifolds with Density and applications} \vspace{0.4cm}

\vspace{0.5cm}

{\large Jos$\acute{\text{e}}$ M. Espinar}\\
\vspace{0.3cm} 

\noindent Instituto de Matem\'{a}tica Pura e Aplicada\\ e-mail: jespinar@impa.br\vspace{0.2cm}

\rule{15cm}{1.5pt}
\end{center}

\vspace{.5cm}

\begin{abstract}

The aim of this paper is twofold. On the one hand, the study of gradient Schr\"{o}dinger operators on manifolds with density $\phi$. We classify the space of solutions when the underlying manifold is $\phi-$parabolic. As an application, we extend the Naber-Yau Liouville Theorem, and we will prove that a complete manifold with density is $\phi -$parabolic if, and only if, it has finite $\phi-$capacity. Moreover, we show that the linear space given by the kernel of a nonnegative gradient Schr\"{o}dinger operators is one dimensional provided there exists a bounded function on it and the underlying manifold is $\phi -$parabolic.

On the other hand, the topological and geometric classification of complete weighted $H_\phi -$stable hypersurfaces immersed in a manifold with density $(\amb , g, \phi)$ satisfying a lower bound on its Bakry-\'{E}mery-Ricci tensor. Also, we classify weighted stable surfaces in a three-manifold with density  whose Perelman scalar curvature, in short, P-scalar curvature, satisfies $\scad + \frac{\abs{\nabla \phi}^2 }{4} \geq 0$. Here, the P-scalar curvature is defined as $\scad = R - 2 \Delta _g \phi - \abs{\nabla _g  \phi  }^2$, being $R$ the scalar curvature of $(\amb ,g)$.

Finally, we discuss the relationship of manifolds with density, Mean Curvature Flow (MCF), Ricci Flow and Optimal Transportation Theory. We obtain classification results for stable self-similiar solutions to the MCF, and also for stable translating solitons to the MCF (as far as we know, the first classification result). Moreover, for gradient Ricci solitons, we recover the Hamilton-Ivey-Perelman classification assuming only a $L^2-$type bound on the scalar curvature and a inequality between the scalar curvature and the Ricci tensor. We also classify gradient Ricci solitons when the scalar curvature does not change sign. We finish by classifying critical transportation plans for the Boltzman entropy on $\phi -$parabolic manifolds. In particular, we recover the case for the Gaussian measure. 
\end{abstract}

%

\vspace{.2cm}

{\bf Keywords:} Manifolds with density, Generalized mean curvature, Stability, Schr\"{o}dinger Operators, Mean Curvature Flow, Ricci Flow, Optimal transport. 

\vspace{.2cm}


\section{Introduction}

There exists a natural measure in a Riemannian manifold $(\amb ,g)$, the Riemannian volume measure $dv_g \equiv dv$. More generally, we can consider smooth metric measure spaces, that is, triples $(\amb , g , m)$ where $m$ is a smooth measure on $\amb$. Equivalently, by the Radon-Nikod\'{y}m Theorem, we can consider triples $(\amb , g , \phi)$, where $\phi \in C^\infty (\amb)$ is a smooth function so that $dm = e^\phi dv$. The triple $(\amb , g , \phi)$ is called a {\bf manifold with density $\phi$}. 

One of the first examples of a manifold with density appeared in the realm of probablity and statistics; the Euclidean Space with the Gaussian density $dm := e^{-\pi |x|^2} \, dx$ (see \cite{ACanVBayFMorCRos08} for a detailed exposition in the context of isoperimetric problems).  In 1985, D. Bakry and M. \'{E}mery \cite{DBakMEme85}  studied manifolds with density in the context of diffusion equations and they introduced the so-called {\bf Bakry-\'{E}mery-Ricci tensor} given by
\begin{equation}\label{Eq:ricd}
\ricd = {\rm Ric} - \camb ^2  \phi , 
\end{equation}where ${\rm Ric}$ and  $\camb ^2$ are the Ricci tensor and  the Hessian with respect to the ambient metric $g$, respectively. Nowadays, manifolds with density appear in many other fields of mathematics.

M. Gromov \cite{MGro03} considered manifolds with density as {\it mm-spaces} and introduced the generalized mean curvature of an oriented hypersurface $\Sigma \subset (\amb ,g , \phi)$, or {\bf weighted mean curvature}, as a natural generalization of the mean curvature. The weighted mean curvature is defined by
\begin{equation*}
H_\phi = H + \metag{N}{\camb  \phi}, 
\end{equation*}where $H$ denotes the usual mean curvature of $\Sigma$, $N$ is the unit normal vector field along $\Sigma$ and $\camb$ is the Gradient operators with respect to the ambient metric $g$. The important point is that hypersurfaces of constant weighted curvature appear as critical points of certain weighted volume functional. This fundamental analogy with constant mean curvature hypersurfaces took the attention on the field (cf. \cite{EBarBSharYWei,VBay03,ACanVBayFMorCRos08,XCheTMejDZho12,GHui90,FMor05} and reference therein).

It is interesting to recall here that  self-similiar solutions to the mean curvature flow in $\r ^{n+1}$ can be seen as weighted minimal hypersurfaces in the Euclidean space endowed with the corresponding density. G. Huisken \cite{GHui90} and T. Colding and W. Minicozzi \cite{TColWMin12a,TColWMin12b} proved this relationship for self-shrinker (and also self-expander). Moreover, T. Illmanen \cite{TIll} showed that translating solitons of the Mean Curvature Flow can be seen as weighted minimal hypersurfaces (see also \cite{XCheTMejDZho12,XCheDZho12,JCluOSchFSch,HLee,GSmi,LWan12} and references therein). 

Another field in Mathematics where manifolds with density appear is in G. Perelman's work \cite{GPer,GPer2} on the Poincar\'{e} conjecture. He was able to formulate the Ricci flow as a gradient flow (cf. \cite{PTop06} for details) for certain {\it Fischer Information functional} $\mathcal F$, and the more general $\mathcal W -$entropy. The aim of Perelman for introducing such functionals was the classification of (gradient) Ricci solitons, that is, self-similar solutions to the Ricci flow. They can be analytically described as manifolds with density $(\amb , g , \phi)$ such that $\ricd = \lambda g$, with $\lambda \in \r$. They are called {\bf shrinking}, {\bf steady} or {\bf expanding} depending if $\lambda > 0$, $\lambda =0$ or $\lambda <0$ respectively. Steady Ricci solitons appear as critical points of the $\mathcal F$ functional (shrinking Ricci solitons appear as critical points of the $\mathcal W-$entropy functional). This is an active topic in Differential Geometry today (cf. \cite{HCao09,TIve93,ANab10,PPetWWyl10,PPetWWyl12}).

In \cite{GPer,GPer2}, Perelman introduced a geometric invariant associatted to the manifold with density that will be important in this paper, the {\bf Perelman Scalar Curvature} $\scad$, in short P-scalar curvature, given by 
\begin{equation}\label{Eq:Scad}
\scad = R - 2 \Delta _g  \phi - \abs{\camb \phi}^2 ,
\end{equation}where $\Delta _g$ is the Laplacian operator with respect to the ambient metric $g$. Perelman proved that $\scad$ is not the trace of $\ricd$ but they are related by a Bianchi type identity. 

Recently, geometric problems on manifolds with density $(\amb , g , \phi) $ has been related to corresponding problems in the Wasserstein space of probability measures equipped with the quadratic Wasserstein metric, which corresponds to a relaxed version of Monge\rq{}s optimal transportation problem (cf. \cite{FOtt01,CVil08} and references therein). 

\section{Organization of the paper and Main results}

We will describe here the organization of the paper and we will outline the main results. 

In Section \ref{Sect:FiniteType} we develop the theory of gradient Schr\"{o}dinger operators acting on piecewise smooth functions of compact support $u \in C^\infty _0 (\Sigma)$ defined on a manifold $\Sigma$. That is, given a Riemannian manifold $(\Sigma ,g )$, complete or compact with boundary possibly empty, $\phi \in C^\infty (\Sigma)$ and $q \in C^\infty (\Sigma) \cap L ^\infty (\Sigma)$ called {\bf potential}, we study differential operators on the form
\begin{equation}\label{Eq:Lphi}
 L_\phi  u := \Delta u + g(\nabla \phi , \nabla u) + q u \, , \, \, u \in C^\infty _0 (\Sigma),
\end{equation}in particular, the $\phi-$Laplacian is defined by $\Delta _\phi u = \Delta u + \metag{\nabla \phi}{ \nabla u}$ occurs when $q \equiv 0$. 

Recall that a manifold with density $(\Sigma , g , \phi) $ is $\phi -$parabolic if there is no nonconstant nonnegative $\phi -$superharmonic function $u$, equivalently, if $\Delta _ \phi u = \Delta u + \metag{\nabla \phi}{\nabla u} \leq 0$ and $u \geq 0$ then $u$ is constant. For any compact set $K \subset \Sigma$, we define the $\phi -$capacity of $K$ as 
$$ {\rm cap}^\phi (K) = {\rm inf}\set{ \int _\Sigma |\nabla u|^2  e^{\phi} dv_\Sigma \, : \, u \in Lip _0 (\Sigma) \text{ and } u _{|K} \equiv 1 } ,$$where $Lip _0 (\Sigma)$ is the set of compactly supported Lipschitz functions on $\Sigma$. 

We will use a notion weaker than the above $\phi -$parabolicity. Namely, we say that $(\Sigma , g , \phi)$ is of {\bf finite $\phi -$capacity} if there exists an exhaustion by relatively compact sets  $ \set{K_i} \subset \Sigma$ with uniformly bounded $\phi -$capacity (see Definition \ref{Def:FiniteType}). Note that Definition \ref{Def:FiniteType} is more general than the notion of $\phi -$parabolicity. Moreover, if $\Sigma$ has dimension $2$, $\Sigma $ is parabolic if, and only if, it is $\phi-$parabolic for any $\phi \in C^\infty (\Sigma)$ by the conformal invariance of the Dirichlet integral. Nevertheless, this is not longer true if ${\rm dim}(\Sigma) \geq 3$.

The following result, which extends \cite[Theorem 2.3]{MManJPerMRod12}, tells us that  gradient Schr\"{o}dinger operators over manifold with finite $\phi -$capacity behaves as in the compact case. 

{\bf Theorem A: } {\it Let $(\Sigma ,g , \phi)$ be a complete manifold with finite $\phi-$capacity and take $q_1,q_2 \in C^\infty (\Sigma)$ such that $q_1 - q_2 \geq 0$ on $\Sigma$. Assume that: 
\begin{enumerate}
\item There exists a positive subsolution $u$ to the gradient Schr\"{o}dinger operator $L_{\phi,q_1} := \Delta + g (\nabla \phi , \nabla \cdot ) +q_1$, i.e., $u$ satisfies
$$ L_{\phi,q_1} u \leq 0 .$$
\item There exists a bounded function $ v \in C^\infty (\Sigma) $ so that 
 $$ L_{\phi,q_2} v := \Delta v + g (\nabla \phi , \nabla v ) +q_2 v \geq 0 .$$
\end{enumerate}

Then, $v/u$ is constant.} 

Moreover, the above technique allows us to extend the Naber-Yau Liouville Theorem \cite{ANab10}, which implies that finite $\phi-$capacity is, in fact, equivalent to $\phi -$parabolicity (see Corollary \ref{Cor:FiniteParabolic}). Specifically,

\begin{quote}
{\bf Theorem \ref{Th:NaberYau}.} {\it  Let $(\Sigma ,g, \phi)$ be a complete manifold with finite $\phi-$capacity. Take $q \in C^\infty (\Sigma)$ and consider the gradient Schr\"{o}dinger operator $L_{\phi,q} := \Delta + g (\nabla \phi , \nabla \cdot ) +q$. 

\begin{enumerate}
 \item Assume $q \geq 0$. Then any nonnegative subsolution $u$ to $L_{\phi, q}$, i.e., $u$ satisfies $ L_{\phi,q} u \leq 0 $, must be constant and either $u\equiv 0$ or $q \equiv 0$ on $\Sigma$.  In particular, any solution bounded above or below to $L _{\phi,q} u= 0$ must be constant.

\item  Assume $q \leq 0$ and $L_{\phi,q}$ is stable. If there is a bounded supersolution $u$, i.e., $u$ satisfies $ L_{\phi,q} u \geq 0 $, then $u$ is constant and either $u \equiv 0$ or $q\equiv 0$ on $\Sigma$.
\end{enumerate}}
\end{quote}

Here, $L_{\phi,q}$ is stable in the sense of Schr\"{o}dinger operators (see Definition \ref{Def:LStableGeneral}), i.e., its first eigenvalue is nonnegative. Hence, we obtain the following characterization of the kernel of a gradient Schr\"{o}dinger operator on $\phi -$parabolic manifolds.

\begin{quote}
{\bf Theorem \ref{Th:Kernel}.} {\it  Let $(\Sigma ,g, \phi)$ be a complete $\phi-$parabolic manifold. Take $q \in C^\infty (\Sigma)$ and assume that the gradient Schr\"{o}dinger operator $L_{\phi,q} := \Delta + g (\nabla \phi , \nabla \cdot ) +q$ is stable. If $u$ is a bounded smooth function in the kernel of $L_{\phi, q}$, then $u$ is either identically zero or it has no zeroes on $\Sigma$. Hence, the linear space of such functions is one dimensional.}
\end{quote}


In Section \ref{Sect:Schrodinger} we classify $H_\phi -$stable hypersurfaces in a manifold with density satisfying a lower bound on the Bakry-\'{E}mery-Ricci tensor applying Theorem A. 

\begin{quote}
{\bf Theorem \ref{Th:Hypersurface}.} {\it Let $\Sigma \subset (\amb , g , \phi)$ be a  $\phi-$parabolic complete $H_\phi -$stable hypersurface and assume that $\ricd \geq k$, $k\geq 0$. Then, $\Sigma$ is totally geodesic and $\ricd (N,N)\equiv 0$ along $\Sigma$. }
\end{quote}

In particular, Theorem \ref{Th:Hypersurface} extends results obtained by X. Cheng, T. Mejia and D. Zhou \cite{XCheTMejDZho12} and G. Liu \cite{GLiu}. The above result is not necessary true if we remove the parabolicity condition, which states the nontriviallity of the above result. Our Theorem \ref{Th:Hypersurface} is the extension to manifolds with density and we must remark that our result does not follows from Manzano-P\'{e}rez-Rodr\'{i}guez's Theorem \cite{MManJPerMRod12}.

We continue this Section \ref{Sect:Schrodinger} by focussing on surfaces. We introduce the key tool for studying $H_\phi -$stable surfaces, that is, we will associate a self-adjoint Schr\"{o}dinger operator to any weighted $H_\phi -$stable surface which is stable in the sense of Schr\"{o}dinger operators. We prove

\begin{quote}
{\bf Main Lemma}. {\it  Let $\Sigma \subset (\amb ^3, g , \phi)$ be a complete weighted $H_\phi -$stable surface. Then, 
\begin{equation*}
\int_ \Sigma( V - a K) f^2dv_\Sigma \leq  \int _\Sigma  |\nabla f|^2  dv_\Sigma 
\end{equation*} for any $f \in C_0 ^\infty (\Sigma)$ compactly supported piecewise smooth function, where 
\begin{equation*}
V:=\frac{1}{3}\left( \frac{1}{2}\scad  +\frac{1}{2}H_\phi ^2+ \frac{1}{2}|A|^2 +\frac{1}{8}|\nabla \phi|^2\right) \text{ and  } a:= \frac{1}{3}.
\end{equation*}

In other words, the Schr\"{o}dinger operator 
$$ L := \Delta - a K + V $$is stable in the sense of Definition \ref{Def:StableL}.}
\end{quote}

From the Main Lemma we continue with the classification of complete weighted $H_\phi -$stable surfaces in a manifold with density $\phi \in C^\infty (\Sigma )$ so that its P-scalar curvature satisfies $\scad + \frac{\abs{\nabla \phi}^2 }{4} \geq c \geq 0$. First, we extend the Rosenberg\rq{}s diameter estimate \cite{HRos06} to weighted $H_\phi-$stable surfaces given a positive lower bound for $\scad + \frac{\abs{\nabla \phi}^2 }{4} + H^2_\phi$.

\begin{quote}
{\bf Theorem \ref{Th:Diameter}.} {\it  Let $\Sigma \subset (\amb , g , \phi )$ be a weighted $H_\phi -$stable surface with boundary $\partial \Sigma$. If $\scad + \frac{\abs{\nabla \phi}^2 }{4} + H^2_\phi \geq c >0$ on $\Sigma$, then 
$$ d_\Sigma \left( p , \partial \Sigma \right) \leq \frac{2\pi }{\sqrt{3c}} \text{ for all } p \in \Sigma ,$$
where $d_\Sigma$ denotes the intrinsic distance in $\Sigma$. Moreover, if $\Sigma$ is complete without boundary, then it must be topologically a sphere.}
\end{quote}

Theorem \ref{Th:Diameter} says that the only complete noncompact weighted $H_\phi -$stable we shall consider are the weighted minimal ones. So, the next step is to classify the complete noncompact weighted minimal stable surfaces in the spirit of the works of D. Fischer-Colbrie and R. Schoen \cite{DFisRSch80} and R. Schoen and S.T. Yau \cite{RSchSYau82}. The theorem we provide here generalizes a previous result of P. T. Ho \cite{PHo10} where he extended Fischer-Colbrie-Schoen result under the additional hypothesis that $|\nabla \phi| $ is bounded and $\scad \geq 0$ on $\Sigma$. Here, we drop the first condition and relax the hypothesis on the P-scalar curvature. Specifically,

\begin{quote}
{\bf Theorem \ref{Th:FCS}.} {\it  Let $\Sigma \subset (\amb , g , \phi )$ be a complete (noncompact) weighted stable minimal surface where $\scad + \frac{\abs{\nabla \phi}^2 }{4}\geq 0$. Then, $\Sigma$ is conformally equivalent either to the complex plane $\c$ or to the cylinder $\s ^2 \times \r$. In the latter case, $\Sigma$ is totally geodesic and flat.}
\end{quote}

\begin{remark}
We must point out that the above results are nontrivial and do not follow from a conformal change of the ambient metric. 
\end{remark}

In Section \ref{Sect:Applications}, we apply the above results to various situations. T. Colding and W. Minicozzi \cite{TColWMin12a,TColWMin12b} proved that there are no $L-$stable self-shrinker $\Sigma$ with polynomial area growth. We say that a self-shrinker (resp. expander)  $\Sigma$ is {\bf $L-$stable} if it is stable as a weighted minimal hypersurface in $(\r ^{n+1} , g_0 , \phi _{-1})$ (resp. $(\r ^{n+1} , g_0 , \phi _{+1})$). Here, $\phi _c = c\frac{|x|^2}{4}$, where $x$ is the position vector in $\r ^{n+1}$, and $c =-1$ if $\Sigma$ is a self-shrinker and $c=+1$ if $\Sigma$ is a self-expander. 

Polynomial area growth implies that 
$$ \int _{B(i+1) \setminus B(i)} e^{\phi _{-1}} \, dv \to 0 ,$$where $B(r)$ denotes the geodesic ball of radius $r$ centered at a fixed point. In particular, $\Sigma$ is  weighted parabolic  (see Definition \ref{Def:SelfFinite}). We say that a complete self-shrinker (resp. expander) is {\bf weighted parabolic} if $(\Sigma , g , \phi _{-1})$ is $\phi _{-1}-$parabolic (resp. $(\Sigma , g ,\phi _{+1})$ is $\phi_{+1}-$parabolic). 

So, applying Theorem A, we can prove:

\begin{quote}
{\bf Theorem \ref{Th:FiniteMCF}.} {\it There are no complete weighted parabolic $L-$stable self-shrinkers in $\r ^{n+1}$.}
\end{quote}

As far as we know, there are no general classification results for $L-$stable self-expanders. We can, at least, give some condition for the nonexistence (see Proposition \ref{Prop:SelfExpander}).

In fact, there is another notion of stability more related to graphs introduced in \cite{TColWMin12a,TColWMin12b}. A self-similar solution to the mean curvature flow is $L_0-$stable when $L_0 u := \Delta  u + \metag{\nabla \phi _c}{ \nabla u} + |A|^2 u =0 .$ is stable in the sense of operators (see Definition \ref{Def:LStableGeneral}). Thus, this leads us to:

\begin{quote}
{\bf Theorem \ref{Th:GraphFlow}.} {\it  The only complete weighted parabolic self-similar solutions (shrinker or expander) to the mean curvature flow that are $L_0-$stable are hyperplanes.

In particular, The only self-similar solutions (shrinker or expander) to the mean curvature flow that are complete multi-graphs and weighted parabolic are hyperplanes. Moreover, the only self-shrinkers that are entire graphs are the hyperplanes.}
\end{quote}

Recently, L. Wang \cite{LWan12} proved that the only self-shrinkers that are entire graphs are hyperplanes. This result follows from ours since a proper self-shrinker has polynomial volume growth \cite[Theorem 1.3]{XCheDZho12}, and hence, it is weighted parabolic. 

Analogously, one can study stable translating solitons. Translating solitons can be seen as weighted minimal hypersurfaces in $(\r ^{n+1} ,\meta{}{} , \phi _T)$, where $\phi _T (x) = \meta{x}{v}$ for a fixed direction $v \in \s ^n$ (cf. \cite{TIll}). So, it makes sense to classify the stable ones as for self-similar solutions. We obtain

\begin{quote}
{\bf Theorem \ref{Th:Translating}:}  {\it There are no complete weighted parabolic stable translating soliton.} 
\end{quote}

As far as we know, this is the only classification result we know for stable translating solitons. Recently, F. Mart\'{i}n, A. Savas-Halilaj and K. Smoczyk \cite{FMar} obtained classification results and topological obstructions for the existence of translating solitons of the mean curvature flow in euclidean space.

We also use Theorem \ref{Th:NaberYau} for classifying gradient Ricci solitons. Noncompact gradient Ricci solitons are fundamaental in the proof of the Poincar\'{e} Conjecture. From the works of R. Hamilton \cite{RHam82}, T. Ivey \cite{TIve93} and G. Perelman \cite{GPer,GPer2} we have: {\it The only three dimensional shrinking gradient Ricci solitons with bounded curvature are the finite quotients of $\r ^3$, $\s ^2 \times \r$ or $\s^3$.} 

P. Petersen and W. Wylie \cite{PPetWWyl10} showed that the only complete shrinking gradient Ricci soliton such that $\int _\Sigma |{\rm Ric}|^2 e^\phi \, dv_\Sigma < \infty $ and vanishing Weyl Tensor, i.e. $W \equiv 0$, are finite quotients of $\r ^n$, $\s ^{n-1} \times \r$ or $\s ^n$ (see \cite{PPetWWyl10,PPetWWyl12} and references therein for a detailed exposition).

We extend the above result imposing only a $L^2-$type bound on the scalar curvature $R$. Moreover, they assume that the Weyl tensor vanishes identically, we change that condition by an inequality between the scalar curvature and the Ricci curvature (see Theorem \ref{Th:Shrinker}).

So we recover the Hamilton-Ivey-Perelman classification assuming only a $L^2-$type bound on the scalar curvature and a inequality between the scalar curvature and the Ricci tensor. We also give the following classification result when the scalar curvature does not change sign:

\begin{quote}
{\bf Theorem \ref{Th:GradientSoliton}.}
{\it \begin{itemize}
\item A complete $\phi -$parabolic gradient expander soliton of nonnegative scalar curvature is Ricci flat.

\item A complete $\phi -$parabolic gradient steady soliton whose scalar curvature does not change sign is Ricci flat. Moreover, if $\phi$ is not constant then it is a product of a Ricci flat manifold with $\r $. Also, $\Sigma$ is diffeomorphic to $\r ^n$.

\item A complete $\phi -$parabolic gradient shrinking soliton of nonpositive scalar curvature is Ricci flat.
\end{itemize}}
\end{quote}

\begin{remark}
All the manifolds and submanifolds will be considered smooth, connected and orientable. Moreover, unless stated otherwise, they will be complete without boundary. 

We should also remark that the required differentiability is much lower of that we consider here, we prefer to avoid technical details assuming smoothness for clarifying the exposition and ideas involved in this paper.

Most of the above results can be generalized to the case that $\Sigma$ has finite index, that is, there is only a finite dimensional space of normal variations which strictly decrease the weighted area (see \cite{EBarBSharYWei}). Moreover, we could relax the hypothesis on the P-scalar curvature by either an integrability condition or a decay condition as in  \cite{PBerPCas12b} or \cite{JEsp12b}.
\end{remark}

\section{Gradient Schr\"{o}dinger operators}\label{Sect:SchrodingerFinite}\label{Sect:FiniteType}

We will follow the references \cite{ICha84,LEva98,DGilNTru83}. Consider $(\Sigma ,g)$ a compact Riemannian manifold with boundary $\partial \Sigma$ (possibly empty). Set $q \in C^\infty (\Sigma)$ and $X \in \campo (\Sigma)$ a smooth vector field along $\Sigma$. Consider the differential linear operator, called {\bf generalized Schr\"{o}dinger operator}, given by 
$$ \begin{matrix} L : & C_0 ^\infty (\Sigma) &\to & C ^\infty (\Sigma) \\
  & u & \to & Lu := \Delta u + g(X, \nabla u) + q \,u ,\end{matrix}$$where $\Delta $ and $\nabla$ are the Laplacian and Gradient with respect to the Riemannian metric $g$ (respectively) and $C_0 ^\infty (\Sigma)$ stands for the linear space of compactly supported piecewise smooth functions on $\Sigma$. We also denote by $\abs{\cdot }$ the norm with respect to the Riemannian metric $g$. Moreover,
\begin{itemize}
\item In the case that $X \equiv \nabla \phi$ for some smooth function $\phi$, i.e., $L _\phi u := \Delta u + g(\nabla \phi , \nabla u) + q u$, we call $L _\phi$ by {\bf gradient Schr\"{o}dinger operator}.

\item In the case that $X \equiv 0$, i.e., $L u := \Delta u  + q u$, we call $L$ simply by {\bf Schr\"{o}dinger operator}.
\end{itemize}

In general, a gradient Schr\"{o}dinger operator is not self-adjoint with respect to the standard $L^2-$inner product because of the first order term, but it is self-adjoint with respect to a weighted inner product, that is, $L_\phi$ is self-adjoint acting on $u \in C^\infty _0 (\Sigma)$ with respect to the weighted inner product given by
$$ \metal{u}{v}_\phi := \int _\Sigma u \, v \, e^\phi \, dv_\Sigma .$$

Given $L_\phi$ a gradient Schr\"{o}dinger operator, we say that $\lambda \in \r $ is an {\bf eigenvalue} of $L_\phi$ if there is a not identically zero function $u \in C^\infty _0 (\Sigma)$ so that 
$$ L_\phi u = - \lambda u . $$

In the case of gradient Schr\"{o}dinger operators, standard spectral theory (see \cite[Pages 335--336]{LEva98}) gives:

\begin{theorem}\label{Th:Evans}
Let $(\Sigma , g) $ be a compact Riemannian manifold with boundary $\partial \Sigma$ (possibly empty). Set $q,\phi \in C^\infty (\Sigma)$ and consider 
$$ L_\phi u := \Delta u + g(\nabla \phi, \nabla u) +qu \, , \, \, u \in C^\infty _0 (\Sigma) .$$

Then, 
\begin{itemize}
\item $L_\phi$ has real eigenvalues $\lambda ^\phi _1 < \lambda ^\phi _2 \leq \ldots $ with $\lambda^\phi _k \to + \infty$ as $k \to + \infty$.

\item There is an orthonormal basis $\set{u_k} \subset C^\infty _0 (\Sigma)$ for the weighted $L^2$-metric $\metal{\cdot}{\cdot}_\phi$ so that $L_\phi u_k = - \lambda ^\phi _k u_k$.

\item The lowest eigenvalue $\lambda ^\phi _1$ is characterized by 
$$ \lambda ^\phi_1 (\Sigma )= {\rm inf}\set{ \dfrac{\int _\Sigma \left( \abs{\nabla u}^2 -q u^2\right) e^\phi \, dv_\Sigma}{ \int_\Sigma u^2 e^\phi \, dv_\Sigma }\, : \, \, u \in C_0^\infty (\Sigma)\setminus \set{0}} .$$

\item Any eigenfunction for $\lambda ^\phi_1$ does not change sign and, consequently, if $u \in C^\infty _0 (\Sigma)$ is another solution to $L_\phi u = -\lambda ^\phi_1 u$, then $u$ is a constant multiple of $u_1$.
\end{itemize}
\end{theorem}

In the case that $(\Sigma ,g)$ is complete and noncompact, $\phi \in C^\infty  (\Sigma)$, there may not be a lowest eigenvalue for $L_\phi$, however, we can still define the bottom of the spectrum (we still call it $\lambda ^\phi _1$) by 

$$ \lambda ^\phi_1 (\Sigma  ) = {\rm inf}\set{ \dfrac{\int _\Sigma \left( \abs{\nabla u}^2 -q u^2\right) e^\phi \, dv_\Sigma}{ \int_\Sigma u^2 e^\phi \, dv_\Sigma }\, : \, \, u \in C_0^\infty (\Sigma)\setminus \set{0}} ,$$where now the infimum is taken over piecewise smooth functions of compact support and we must allow $\lambda ^\phi_1 (\Sigma )= -\infty$. Another way to characteize $\lambda ^\phi _1 (\Sigma )$ is
$$ \lambda ^\phi _1 (\Sigma ) = {\rm inf}\set{  \lambda ^\phi_1 (\Omega ) \, : \, \, \Omega \subset \Sigma },$$where the infimum is taken over all relatively compact domains $\Omega \subset \Sigma$ with (at least) $C^1 $ boundary. 

\begin{definition}\label{Def:LStableGeneral}
Let $(\Sigma , g)$ be a complete manifold. Set $\phi \in C^\infty  (\Sigma)$ and consider $L_\phi$ a gradient Schr\"{o}dinger operator acting on $u \in C^\infty _0 (\Sigma)$. We say that $L_\phi$ is {\bf stable} if $\lambda ^\phi _1 (\Sigma ) \geq 0$.
\end{definition}

We continue by characterizing stable gradient Schr\"{o}dinger operators in terms of solutions to the equation $L_\phi u = -\lambda _1 ^\phi (\Sigma)$ by a variation of an argument of Fischer-Colbrie and Schoen. T. Colding and W. Minicozzi \cite{TColWMin12a,TColWMin12b} proved a particular case of the following result in the context of stable self-shrinkers. We will prove here the general version that will be important in what follows.

\begin{lemma}\label{Lem:Characterization}
Let $(\Sigma ,g)$ be a complete manifold and $\phi \in C^\infty (\Sigma)$. Then, the following statements are equivalents: 

\begin{enumerate}
\item[(1)] $L_\phi$ is stable.
\item[(2)] There exists a positive function $u$ on $\Sigma$ so that $L_\phi u = - \lambda^\phi _1 (\Sigma) u$, $\lambda ^\phi_1 (\Sigma ) \geq 0$.
\end{enumerate}
\end{lemma}
\begin{proof}
{\bf (1) implies (2):} We argue as in \cite[Lemma 9.25]{TColWMin12a}. Since $L_\phi$ is stable, Definition \ref{Def:LStableGeneral} yields that $\lambda ^\phi_1 (\Sigma ) \geq 0$. 

Fix a point $p \in \Sigma$ and let $B(r_k)$ be the geodesic ball in $\Sigma$ centered at $p$ of radius $r_k$. We can consider an increasing sequence $\set{r_k}$ so that $\partial B(r_k)$ is at least $C^1$ and $r_k \to + \infty$ as $k \to +\infty$.

It is clear that $\lambda ^\phi_1 (B(r_k)) > \lambda ^\phi_1 (\Sigma ) \geq 0$ and 
$$ \lim _{k \to +\infty} \lambda ^\phi_1 (B(r_k)) = \lambda ^\phi_1 (\Sigma ) \geq 0 . $$

From Theorem \ref{Th:Evans}, there exists a positive Dirichlet function $u_k$ on $B(r_k)$ so that 
$$ L_\phi u_k = - \lambda ^\phi _1 (B(r_k)) u_k ,$$so that (after multiplying by a constant) $u_k (p) =1$. This holds for all $k$.

Then, the sequence $\set{u_k}$ is uniformly bounded on compact sets of $\Sigma$ by the Harnack Inequality \cite[Theorem 8.20]{DGilNTru83}. Also, $\set{u_k}$ has all its derivatives uniformly bounded on compact subsets of $\Sigma$ by Schauder estimates \cite[Theorem 6.2]{DGilNTru83}. Therefore, Arzela-Ascoli's Theorem and a diagonal argument give us that a subsequence (that we still denote by $\set{u_k}$) converges on compact subsets of $\Sigma$ to a function $u \in C^\infty (\Sigma)$ which satisfies
$$ \left\{ \begin{matrix} 
L_\phi u = -\lambda ^\phi_1 (\Sigma ) u & \text{ in } & \Sigma \\
u \geq 0               & \text{ in } & \Sigma \\
u (p) = 1               & \text{ at } & p \in \Sigma ,
\end{matrix}\right. $$again, by the Harnack Inequality, we obtain that $u >0 $ in $\Sigma$. This proves (1) implies (2). 

\vspace{.3cm}

{\bf (2) implies (1):} We argue here as in \cite[Proposition 3.2]{TColWMin12b}. Let $u$ be the positive smooth function given by item (2). Then, it satisfies $L_\phi u \leq 0$ since $\lambda ^\phi_1 (\Sigma) \geq 0$. Set $w = \ln u$, then
$$ \Delta w =  \frac{\Delta u}{u} - \abs{\nabla w}^2 \leq - g(\nabla \phi , \nabla w) -q - \abs{\nabla w} ^2 ,$$which implies
$$ {\rm div}\left( e^\phi \nabla w \right) \leq  -\left( q + \abs{\nabla w}^2 \right)e^\phi .$$

Given $\psi \in C_0 ^\infty (\Sigma)$, applying Stokes\rq{} Theorem to ${\rm div}\left( \psi ^2 e^\phi \nabla w\right)$ gives
\begin{equation*}
\begin{split}
0 &= \int _\Sigma \left( 2 \psi e^\phi g(\nabla \psi , \nabla w) +\psi ^2 {\rm div}\left( e^\phi \nabla w\right)\right)  dv_\Sigma  \\
 &\leq \int _\Sigma \left( 2 \psi g(\nabla \psi , \nabla w) -\psi ^2 \left(q + \abs{\nabla w}^2\right)\right) e^\phi dv_\Sigma \\
 & \leq \int _\Sigma \left( \abs{\nabla \psi}^2 -\psi ^2 q \right)  e^\phi dv_\Sigma
\end{split}
\end{equation*}where we have used $2 \psi g(\nabla \psi , \nabla w) \leq  \abs{\nabla \psi}^2 + \psi ^2  \abs{\nabla w}^2$. Therefore, $\lambda ^\phi _1 (\Sigma ) \geq 0$.
\end{proof}

We can see from Theorem \ref{Th:Evans} that, if $\Sigma$ is compact and there exists a non identically zero solution of $L_\phi u =0$, $\phi \in C^\infty (\Sigma)$, then $u$ vanishes nowhere and the linear space of such functions is one dimensional. 

In the case $\nabla \phi \equiv 0$, J. M. Manzano, J. P\'{e}rez and M. M. Rodr\'{i}guez \cite{MManJPerMRod12} proved that the above property extends to the noncompact case when we assume the underlying manifold has finite capacity. Specifically, they proved that if $u$ is a bounded smooth function in the kernel of a nonnegative Schr\"{o}dinger operator $L := \Delta +q$ on a complete manifold of finite capacity, then $u$ is either identically zero or it has no zeroes, moreover the linear space of such functions is one dimensional.

We would like to extend this property to gradient Schr\"{o}dinger operators. First, recall that 
\begin{definition}\label{Def:parabolic}
A manifold with density $(\Sigma , g , \phi) $ is $\phi -$parabolic if there is no nonconstant nonnegative $\phi -$superharmonic function $u$, equivalently, if $\Delta _ \phi u = \Delta u + \metag{\nabla \phi}{\nabla u} \leq 0$ and $u \geq 0$ then $u$ is constant. 
\end{definition}

For any compact set $K \subset \Sigma$, we define the $\phi -$capacity of $K$ as 
$$ {\rm cap}^\phi (K) = {\rm inf}\set{ \int _\Sigma |\nabla u|^2  e^{\phi} dv_\Sigma \, : \, u \in Lip _0 (\Sigma) \text{ and } u _{|K} \equiv 1 } ,$$where $Lip _0 (\Sigma)$ is the set of compactly supported Lipschitz functions on $\Sigma$. Hence, we have the following characterization of $\phi - $parabolicity (cf. \cite{HuaLiuXia,AGri99})

\begin{proposition}
Let $(\Sigma , g , \phi )$ be a complete manifold with density. Then, the following are equivalent:
\begin{enumerate}
\item $\Sigma $ is $\phi -$parabolic.
\item ${\rm cap}^\phi (K)$ for some (then any) compact set $K \subset \Sigma$.
\item Any bounded $\phi -$superharmonic function on $\Sigma$ is constant. 
\end{enumerate}
\end{proposition}

Hence, this leads us to define

\begin{definition}\label{Def:FiniteType}
Let $(\Sigma ,g)$ be a complete manifold and $\phi \in C^\infty (\Sigma)$. We say that $(\Sigma , g , \phi)$ has {\bf finite $\phi$-capacity} if there exists a sequence of cut-off functions $\set{\psi _i}_i \subset C^\infty _0 (\Sigma)$ such that 
\begin{itemize}
\item $0 \leq \psi _i \leq 1 $ in $\Sigma$.

\item The compact sets $\Omega _i := \psi _i ^{-1}(1)$ form an increasing exhaustion of $\Sigma$.

\item The sequence of weighted energies $\set{ \int _\Sigma |\nabla \psi _i|^2 e^\phi \, dv_\Sigma }_i$ is bounded. 
\end{itemize}
\end{definition}

In the case $\nabla \phi =0$, Definition \ref{Def:FiniteType} says that $(\Sigma , g)$ has finite capacity (see \cite{AGri99}). Moreover, if $\Sigma$ has dimension $2$, $\Sigma $ is parabolic if, and only if, it is $\phi-$parabolic for any $\phi \in C^\infty (\Sigma)$ by the conformal invariance of the Dirichlet integral. Nevertheless, this is not longer true if ${\rm dim}(\Sigma) \geq 3$.

The following result, which extends \cite[Theorem 2.3]{MManJPerMRod12}, tells us that gradient Schr\"{o}dinger operators acting on function defined on a $\phi-$parabolic manifold $\Sigma$ behaves as in the compact case.

\begin{quote}
{\bf Theorem A: } {\it Let $(\Sigma ,g , \phi)$ be a complete manifold with finite $\phi-$capacity and take $q_1,q_2 \in C^\infty (\Sigma)$ such that $q_1 - q_2 \geq 0$ on $\Sigma$. Assume that: 
\begin{enumerate}
\item There exists a positive subsolution $u$ to the gradient Schr\"{o}dinger operator $L_{\phi,q_1} := \Delta + g (\nabla \phi , \nabla \cdot ) +q_1$, i.e., $u$ satisfies
$$ L_{\phi,q_1} u \leq 0 .$$
\item There exists a bounded function $ v \in C^\infty (\Sigma) $ so that 
 $$ L_{\phi,q_2} v := \Delta v + g (\nabla \phi , \nabla v ) +q_2 v \geq 0 .$$
\end{enumerate}

Then, $v/u$ is constant.} 
\end{quote}
\begin{proof}
Let $u,v \in C^\infty (\Sigma)$ as in the statement of the theorem. A straightforward computation shows that 
\begin{equation}\label{Eq:Div}
\frac{v}{u} {\rm div}\left( u^2 e^\phi \nabla (v/u)\right) = v \, {\rm div}(e^\phi \nabla v) - \frac{v^2}{u}{\rm div}(e^\phi \nabla u).
\end{equation}

Now, since $u$ and $v$ satisfies the differential inequalities in the statement of the theorem, multiplying by $e^\phi$, we get
\begin{equation*}
{\rm div}(e^\phi \nabla u) \leq -q_1 u e^\phi \quad \text{ and } \quad {\rm div}(e^\phi \nabla v) \geq -q_2 v e^\phi ,
\end{equation*}with this information and \eqref{Eq:Div}, we obtain
\begin{equation}\label{Eq:Div2}
\frac{v}{u} {\rm div}\left( u^2 e^\phi \nabla (v/u)\right) \geq (q_1 -q_2 )v^2 e^\phi \geq 0 .
\end{equation}

Consider the sequence $\set{\psi _i}_i \subset C^\infty _0 (\Sigma)$ given by Definition \ref{Def:FiniteType}. Then, 
\begin{equation*}
\begin{split}
 0 &= \int_ \Sigma {\rm div}\left( (\psi ^2 _i \frac{v}{u}) ( u^2 e^\phi \nabla (v/u))\right) dv_\Sigma \\
 &=  \int_ \Sigma \psi ^2 _i \frac{v}{u}{\rm div}\left(  u^2 e^\phi \nabla (v/u)\right) dv_\Sigma + \int_ \Sigma g\left(  \nabla \left(\psi ^2 _i \frac{v}{u}\right) , u^2 e^\phi \nabla (v/u)\right) dv_\Sigma \\
& \geq 2 \int _\Sigma \psi _i u v e^\phi g\left(  \nabla \psi  _i  ,  \nabla (v/u)\right) dv_\Sigma + \int _\Sigma \psi ^2_ i u^2 e^\phi \abs{\nabla (v/u)}^2 \,dv_\Sigma ,
\end{split}
\end{equation*}that is, using H\"{o}lder Inequality,
\begin{equation*}
\begin{split}
 \int _\Sigma \psi ^2_ i u^2 e^\phi \abs{\nabla (v/u)}^2 \,dv_\Sigma & \leq - 2 \int _\Sigma \psi _i u v e^\phi g\left(  \nabla \psi  _i  ,  \nabla (v/u)\right) dv_\Sigma \\
 & = - 2 \int _{supp(\psi _i) \setminus \Omega _i} \psi _i u v e^\phi g\left(  \nabla \psi  _i  ,  \nabla (v/u)\right) dv_\Sigma \\
 & \leq 2 \left( \int _{supp(\psi _i) \setminus \Omega _i} \psi ^2_ i u^2 e^\phi \abs{\nabla (v/u)}^2 \,dv_\Sigma\right)^{\frac{1}{2}} \left( \int _{supp(\psi _i) \setminus \Omega _i}  v^2 e^\phi \abs{\nabla \psi _i}^2 \,dv_\Sigma\right)^{\frac{1}{2}} \\
 & \leq 2 \left( \int _\Sigma \psi ^2_ i u^2 e^\phi \abs{\nabla (v/u)}^2 \,dv_\Sigma\right)^{\frac{1}{2}} \left( \int _\Sigma  v^2 e^\phi \abs{\nabla \psi _i}^2 \,dv_\Sigma\right)^{\frac{1}{2}} .
\end{split}
\end{equation*}

Now, since $(\Sigma ,g, \phi)$ has finite $\phi -$capacity and $v$ is bounded, there exists a constant $C$ so that 
$$ \int _\Sigma  v^2 e^\phi \abs{\nabla \psi _i}^2 \,dv_\Sigma \leq C \text{ for all } i ,$$which implies from the previous inequality that
\begin{equation*}
\int _\Sigma \psi ^2_ i u^2 e^\phi \abs{\nabla (v/u)}^2 \,dv_\Sigma  \leq 4 C \text{ for all } i,
\end{equation*}
therefore, we obtain  
\begin{equation}\label{Eq:Bound}
\int _\Sigma  u^2 e^\phi \abs{\nabla (v/u)}^2 \,dv_\Sigma  \leq 4 C 
\end{equation}and
\begin{equation}\label{Eq:Bound2}
\lim _{i\to + \infty }\int _{supp(\psi _i) \setminus \Omega _i} \psi ^2_ i u^2 e^\phi \abs{\nabla (v/u)}^2 \,dv_\Sigma  =0 .
\end{equation}

Thus, \eqref{Eq:Bound} and \eqref{Eq:Bound2} imply that
\begin{equation*}
\int _{\Omega _ i} u^2 e^\phi \abs{\nabla (v/u)}^2 \,dv_\Sigma \leq 2\sqrt{C} \left( \int _{supp(\psi _i) \setminus \Omega _i} \psi ^2_ i u^2 e^\phi \abs{\nabla (v/u)}^2 \,dv_\Sigma\right)^{\frac{1}{2}}  \to 0
\end{equation*}as $i$ goes to $+\infty$, thus $v/u$ is constant on $\Sigma$.
\end{proof}

As a consequence of the above Theorem A, we can extend the Naber-Yau Liouville Theorem (see \cite{PPetWWyl10}) to manifold with density of finite $\phi -$capacity.

\begin{theorem}\label{Th:NaberYau}
Let $(\Sigma ,g, \phi)$ be a complete complete manifold with finite $\phi-$capacity. Take $q \in C^\infty (\Sigma)$ and consider the gradient Schr\"{o}dinger operator $L_{\phi,q} := \Delta + g (\nabla \phi , \nabla \cdot ) +q$.

\begin{enumerate}
 \item Assume $q \geq 0$. Then any nonnegative subsolution $u$ to $L_{\phi, q}$, i.e., $u$ satisfies $ L_{\phi,q} u \leq 0 $, must be constant and either $u\equiv 0$ or $q \equiv 0$ on $\Sigma$.  In particular, any solution bounded above or below to $L _{\phi,q} u= 0$ must be constant.

\item  Assume $q \leq 0$ and $L_{\phi,q}$ is stable. If there is a bounded supersolution $u$, i.e., $u$ satisfies $ L_{\phi,q} u \geq 0 $, then $u$ is constant and either $u \equiv 0$ or $q\equiv 0$ on $\Sigma$.
\end{enumerate}
\end{theorem}
\begin{proof}

Let us prove each case:

\begin{enumerate}
\item Assume $q \geq 0$. A straightforward computation, as above, shows that 
\begin{equation*}
\begin{split}
0 &= \int _\Sigma {\rm div}\left( \psi _i ^2 e^{\phi-u} \nabla u \right) dv_\Sigma \\
  &= \int _\Sigma\psi _i ^2 e^{-u} {\rm div}\left(  e^\phi \nabla u \right) dv_\Sigma +\int _\Sigma g\left( \nabla (\psi _i ^2 e^{-u}),  e^\phi \nabla u \right) dv_\Sigma \\
 & \leq 2\int _\Sigma \psi _i e^{\phi-u} g\left( \nabla \psi _i ,  \nabla u \right) dv_\Sigma -\int _\Sigma \psi _i ^2 e^{\phi -u} \abs{ \nabla u }^2 dv_\Sigma -\int _\Sigma \psi _i ^2 e^{\phi -u} u q \, dv_\Sigma,
\end{split}
\end{equation*}therefore, using H\"{o}lder Inequality and $u\geq 0$, we obtain
\begin{equation*}
\begin{split}
 \int _\Sigma \psi _i ^2 e^{\phi-u} \abs{ \nabla u }^2 dv_\Sigma &\leq  \int _\Sigma \psi _i ^2 e^{\phi -u} \abs{ \nabla u }^2 dv_\Sigma +\int _\Sigma \psi _i ^2 e^{\phi -u} u q \, dv_\Sigma \\
 &\leq  2 \left( \int _{supp(\psi _i) \setminus \Omega _i} \psi ^2_i  e^{\phi -u} \abs{\nabla u}^2 \,dv_\Sigma\right)^{\frac{1}{2}} \left( \int _{supp(\psi _i) \setminus \Omega _i}   e^{\phi -u} \abs{\nabla \psi _i}^2 \,dv_\Sigma\right)^{\frac{1}{2}} \\
 & \leq  2 \left( \int _{supp(\psi _i) \setminus \Omega _i} \psi ^2_i  e^{\phi -u} \abs{\nabla u}^2 \,dv_\Sigma\right)^{\frac{1}{2}} \left( \int _{supp(\psi _i) \setminus \Omega _i}   e^\phi \abs{\nabla \psi _i}^2 \,dv_\Sigma\right)^{\frac{1}{2}}.
\end{split}
\end{equation*}

Thus, arguing as Theorem A, we obtain that $u$ must be constant and either $u\equiv 0$ or $q \equiv 0$.

\item Assume $q \leq 0$ and $u$ is bounded. Then, Theorem A implies that $u>0$. Therefore, calculating as above: 
\begin{equation*}
\begin{split}
0 &= \int _\Sigma {\rm div}\left( \psi _i ^2 u e^\phi \nabla u \right) dv_\Sigma \\
  &= \int _\Sigma\psi _i ^2 u \, {\rm div}\left(  e^\phi \nabla u \right) dv_\Sigma +\int _\Sigma e^\phi g\left( \nabla (\psi _i ^2 u), \nabla u \right) dv_\Sigma \\
 & \geq 2\int _\Sigma \psi _i u e^\phi g\left( \nabla \psi _i ,  \nabla u \right) dv_\Sigma + \int _\Sigma \psi _i ^2 e^\phi \abs{ \nabla u }^2 dv_\Sigma - \int _\Sigma \psi _i ^2 e^\phi u^2 q \, dv_\Sigma,
\end{split}
\end{equation*}that is, 
$$ \int _\Sigma \psi _i ^2 e^\phi \abs{ \nabla u }^2 dv_\Sigma - \int _\Sigma \psi _i ^2 e^\phi u^2 q dv_\Sigma \leq - 2\int _\Sigma \psi _i u e^\phi g\left( \nabla \psi _i ,  \nabla u \right) dv_\Sigma ,$$and we can finish as in the previous case.
\end{enumerate}
\end{proof}

It is important to remark here that Theorem A, and hence Theorem \ref{Th:NaberYau}, does not follow from \cite{MManJPerMRod12} by doing a conformal change of metric when ${\rm dim}(\Sigma)\geq 3$.

In fact, we can relax a little the hypothesis on $u$ in item $2.$ of Theorem \ref{Th:NaberYau}, namely

\begin{corollary}\label{Cor:NaberYau}
Let $(\Sigma ,g, \phi)$ be a complete manifold with finite $\phi-$capacity. Take $q \in C^\infty (\Sigma)$ and consider the gradient Schr\"{o}dinger operator $L_{\phi,q} := \Delta + g (\nabla \phi , \nabla \cdot ) +q$, where $q\leq 0$.

Assume there is a nonnegative supersolution $u$, i.e., $u$ satisfies $ L_{\phi ,q} u \geq 0 $, such that the sequence $\set{\int _\Sigma \abs{\nabla \psi _i}^2u^2 e^\phi \, dv_\Sigma }_i$ is uniformly bounded. 

Then $u$ is constant and either $u \equiv 0$ or $q\equiv 0$ on $\Sigma$. Moreover, if $q \equiv 0$, we only need to assume that $u$ is bounded below.
\end{corollary}
\begin{proof}
We argue as in item 2. That is, we have 
$$ \int _\Sigma \psi _i ^2 e^\phi \abs{ \nabla u }^2 dv_\Sigma - \int _\Sigma \psi _i ^2 e^\phi u^2 q dv_\Sigma \leq - 2\int _\Sigma \psi _i u e^\phi g\left( \nabla \psi _i ,  \nabla u \right) dv_\Sigma ,$$which implies
$$ \int _\Sigma \psi _i ^2 e^\phi \abs{ \nabla u }^2 dv_\Sigma  \leq 2 \left( \int _{supp(\psi _i) \setminus \Omega _i}  \psi _i ^2 e^\phi \abs{ \nabla u }^2 dv_\Sigma \right)^{1/2} \left( \int _{supp(\psi _i) \setminus \Omega _i}  u ^2 e^\phi \abs{ \nabla \psi _i }^2 dv_\Sigma\right)^{1/2} , $$therefore, since the sequence $\set{\int _\Sigma \abs{\nabla \psi _i}^2u^2 e^\phi \, dv_\Sigma }_i$ is uniformly bounded, we can finish as in Theorem A.
\end{proof}

Moreover, Theorem \ref{Th:NaberYau} also proves

\begin{corollary}\label{Cor:FiniteParabolic}
Let $(\Sigma ,g, \phi)$ be a complete manifold with density. It has finite $\phi -$capacity if, and only if, it is $\phi -$parabolic.
\end{corollary}

And, as a consequence of Theorem A and Corollary \ref{Cor:FiniteParabolic}, we obtain
\begin{theorem}\label{Th:Kernel}
Let $(\Sigma ,g, \phi)$ be a complete $\phi-$parabolic manifold. Take $q \in C^\infty (\Sigma)$ and assume that the gradient Schr\"{o}dinger operator $L_{\phi,q} := \Delta + g (\nabla \phi , \nabla \cdot ) +q$ is stable. 

If $u$ is a bounded smooth function in the kernel of $L_{\phi, q}$, then $u$ is either identically zero or it has no zeroes on $\Sigma$. Hence, the linear space of such functions is one dimensional.
\end{theorem}

\section{Weighted stable surfaces and gradient Schr\"{o}dinger operators}\label{Sect:Schrodinger}

M. Gromov \cite{MGro03} considered manifolds with density as {\it mm-spaces} and introduced the generalized mean curvature, or {\bf weighted mean curvature}, of an oriented hypersurface $\Sigma \subset (\amb ,g , \phi)$  as a natural generalization of the mean curvature. The weighted mean curvature is defined by
\begin{equation}\label{Eq:Hd}
H_\phi = H + \metag{N}{\camb  \phi}, 
\end{equation}where $H$ denotes the usual mean curvature of $\Sigma$ and $N$ is the unit normal vector field along $\Sigma$.

\begin{remark}
In this paper, the mean curvature will be just the trace, and not the mean of the trace. Moreover, we assume the ambient manifold and the hypersurfaces immersed in it are orientables. 
\end{remark}

Let $(\amb , g, \phi)$ be a manifold with density $\phi \in C^{\infty}(\amb)$. For any Borel set $\Omega \subset \amb$ with boundary $\Sigma := \partial \Omega$ and inward unit normal $N$ along $\Sigma$, the {\bf weighted volume} of $\Omega$ and the {\bf weighted area} of $\partial \Omega$ are given by
$$ {\rm V}_\phi (\Omega) := \int _{\Omega} e^\phi \, dv  , \, \, \, {\rm A}_\phi (\Sigma) := \int _{\Sigma} e^\phi \, dv_\Sigma ,$$where $dv$ and $dv_\Sigma$ are the volume element and area element with respect to $g$ respectively. 

\begin{remark}
Note that,  when $\phi \equiv 0$,  ${\rm V}_\phi$ and ${\rm A}_\phi$ are nothing but the usual volume and area with respect to $g$. 
\end{remark}

Let $\Psi _t$, $t \in (-\eps, \eps)$, be a smooth family of diffeomorphisms of $\amb$ so that $\Psi _0 $ is the identity. Assume
$$ \derb{}{t}{t=0} \left. \Psi _t \right. _{| \Sigma} = X + u N,$$where $X$ is a tangential vector field along $\Sigma $ and $u$ is a smooth function with compact support on $\Sigma$. An immersed hypersurface $\Sigma \subset (\amb , g, \phi )$ is {\bf weighted minimal}, in short $\phi-$minimal, if it is a critical point of the weighted area functional, i.e., if 
$$ \derb{}{t}{t=0}  A_\phi (\Psi _t (\Sigma)) = 0 ,$$for all compactly supported variations. From Bayle's variational formulae \cite{VBay03} we can see that 
$$ \derb{}{t}{t=0}  A_\phi (\Psi _t (\Sigma)) = \int _\Sigma H_\phi \, u \, e^\phi  \, dv_\Sigma ,$$that is, 

\begin{definition}\label{Def:Stationary}
$\Sigma $ is {\bf $\phi -$minimal} if and only if the weighted mean curvature vanishes. More generally, an immersed hypersurface $\Sigma$ has {\bf constant weighted mean curvature} $H_\phi = H_0$ if and only if 
$$ \derb{}{t}{t=0}  \left( A_\phi (\Psi _t (\Sigma)) -H_0 \, V_\phi (\Psi _t (\Sigma)) \right)= 0$$for all compactly supported variations $\set{\Psi _t}$ (cf. \cite{ACanVBayFMorCRos08}).
\end{definition}

As a first sight, the set of all constant weighted mean curvature hypersurfaces is too big. So, we can focus on those which (locally) minimize the functional 
$$F_\phi (\Psi _t(\Sigma)):=  A_\phi (\Psi _t (\Sigma)) -H_\phi \, V_\phi (\Psi _t (\Sigma))  $$up to second order, that is, weighted stationary domains with nonnegative second variation for all compactly 
supported variations.

\begin{definition}\label{Def:StableL}
We say that a constant weighted mean curvature hypersurface $\Sigma \subset (\amb ,g, \phi)$ is {\bf stable} if 
$$ \derb{\mbox{}^2}{t^2}{t=0}  \left( A_\phi (\Psi _t (\Sigma)) -H_\phi V_\phi (\Psi _t (\Sigma)) \right) \geq  0,$$for all compactly supported variations $\set{\Psi _t}$. In this case, we say that $\Sigma$ is {\bf weighted $H_\phi -$stable}, in short, {\bf $H_\phi-$stable}.
\end{definition}

Here, we are adopting the strong notion of stability and not the weak version that deals with isoperimetric problems (see \cite{ACanVBayFMorCRos08} for details). From Bayle's variational formula \cite{VBay03}, one can prove:

\begin{lemma}\label{Lem:SecondVariation}
Consider a weighted stationary open set $\Omega \subset (\amb , g, \phi )$. Let $N$ be the inward unit normal along $\Sigma = \partial \Omega$, and $H_\phi$ the constant weighted mean curvature of $\Sigma$ with respect to $N$. Consider a compactly supported variation of $\Omega$, $\Psi _t$, with normal part $u N$ on $\Sigma$. Then, we have
\begin{equation}\label{Eq:2Variation1}
\int_\Sigma \left( \ricd (N,N) + |A|^2\right)u ^2 e^\phi  \, dv_\Sigma \leq \int _\Sigma |\nabla u|^2 e^\phi \, dv_\Sigma ,
\end{equation}where $|A|^2$ is the squared norm of the second fundamental form and $\ricd $ is the Bakry-\'{E}mery-Ricci curvature.

Equivalently, we have:
\begin{equation}\label{Eq:2Variation}
F_\phi '' (0) = - \int _{\Sigma}  u \, L _\phi u \, e^\phi \, dv_\Sigma ,
\end{equation}where
\begin{equation}\label{Eq:Lphi}
L_\phi u := \Delta u + \metag{\nabla  \phi }{\nabla u} + (|A|^2 + \ricd (N,N))u .
\end{equation}

Here $\Delta$ and $\nabla $ are the Laplacian and Gradient operators with respect to the induced metric on $\Sigma$, $A$ is the second fundamental form of $\Sigma$, $\ricd $ is the Bakry-\'{E}mery-Ricci tensor of $(\amb , g, \phi )$.
\end{lemma}

As we can see from Lemma \ref{Lem:SecondVariation}, a weighted $H_\phi -$stable surface has a gradient Schr\"{o}dinger operator, $L_\phi$, given by \eqref{Eq:Lphi}, associated to it. Hence, it is clear from the definitions that

\begin{proposition}
 $\Sigma$ is a weighted $H_\phi -$stable in the sense of Definition \ref{Def:StableL} if, and only if, gradient Schr\"{o}dinger operator, $L_\phi$, given by \eqref{Eq:Lphi} is stable in the sense of Definition \ref{Def:LStableGeneral}
\end{proposition}

\subsection{Complete weighted stable hypersurfaces}

We continue by classifying stable hypersurfaces in manifolds with density that verify a lower bound on its Bakry-\'{E}mery-Ricci tensor that extends a recent result by G. Liu \cite{GLiu}. Specifically:

\begin{theorem}\label{Th:Hypersurface} 
Let $\Sigma \subset (\amb , g , \phi)$ be a $\phi -$parabolic complete $H_\phi -$stable hypersurface and assume that $\ricd \geq k$, $k\geq 0$. Then, $\Sigma$ is totally geodesic and $\ricd (N,N)\equiv 0$ along $\Sigma$. 
\end{theorem}
\begin{proof}
Since $\Sigma \subset (\amb , g , \phi) $ is $H_\phi -$stable, the gradient Schr\"{o}dinger operator, given by \eqref{Eq:Lphi} 
$$L_\phi u := \Delta u + \metag{\nabla \phi}{\nabla u} + q _\phi u,$$ where $q_\phi := \ricd (N,N) + |A|^2$, is stable in the sense of Definition \ref{Def:StableL}. Therefore, from Lemma \ref{Lem:Characterization}, there exists a positive subsolution, i.e., there exists $u>0$ such that $L_\phi u \leq 0$. Moreover, $q _\phi \geq 0$ since $\ricd \geq 0$. 

Thus, Theorem \ref{Th:NaberYau} implies that $q_\phi \equiv 0$, that is, $\Sigma$ is totally geodesic and $\ricd (N,N) \equiv 0$ along $\Sigma$.
\end{proof}

We shall remark the difference with the standard (unweighted) case. In \cite{DFisRSch80}, D. Fischer-Colbrie and R. Schoen proved that a complete stable minimal surface immersed in a manifold with nonnegative Ricci curvature must be totally umbilic and the Ricci curvature in the normal direction must vanish along the surface. Manzano-P\'{e}rez-Rodr\'{i}guez \cite{MManJPerMRod12} extended the above result for hypersurfaces

\begin{quote}
{\bf Theorem \cite{MManJPerMRod12} :} {\it Any complete parabolic two-sided stable minimal hypersurface in a manifold of nonnegative Ricci curvature must be totally geodesic and the Ricci curvature in the normal direction must vanish along the hypersurface.} 
\end{quote}

The above result is not necessary true if we remove the parabolicity condition, which states the nontriviallity of the above result. Our Theorem \ref{Th:Hypersurface} is the extension to manifolds with density and we must remark that our result does not follows from Manzano-P\'{e}rez-Rodr\'{i}guez's Theorem.

\subsection{Complete weighted stable surfaces}

From this point on, we will focus on two dimensional surfaces $\Sigma$ immersed in a three-dimensional manifold with density $(\amb , g ,\phi)$. In G. Perelman's work \cite{GPer} on the Poincar\'{e} conjecture, Perelman introduced the geometric quantity $\scad$, called the {\bf Perelman Scalar Curvature}, in short P-scalar curvature, and it given by 
\begin{equation}\label{Eq:Scad}
\scad = R - 2 \Delta _g  \phi - \abs{\camb \phi}^2 ,
\end{equation}where $\Delta _g$ and $\camb$ are the Laplacian and Gradient operators with respect to the ambient metric $g$ respectively. Moreover, Perelman proved that $\scad$ is not the trace of $\ricd$ but they are related by a Bianchi type identity. 

Now, we will see proof the Main Lemma in this section, we will associate a self-adjoint Schr\"{o}dinger operator $L_0$ to a weighted $H_\phi -$stable surface. 

\begin{quote}
{\bf Main Lemma: }{\it Let $\Sigma \subset (\amb ^3, g , \phi)$ be a complete weighted $H_\phi -$stable surface. Then, 
\begin{equation}\label{LDensity}
\int_ \Sigma( V - a K) f^2dv_\Sigma \leq  \int _\Sigma  |\nabla f|^2  dv_\Sigma 
\end{equation} for any $f \in C_0 ^\infty (\Sigma)$ compactly supported piecewise smooth function, where 
\begin{equation}\label{V}
V:=\frac{1}{3}\left( \frac{1}{2}\scad  +\frac{1}{2}H_\phi ^2+ \frac{1}{2}|A|^2 +\frac{1}{8}|\nabla  \phi|^2\right) \text{ and  } a:= \frac{1}{3}.
\end{equation}

In other works, the Schr\"{o}dinger operator 
$$ L_0 := \Delta - a K + V $$is stable in the sense of Definition \ref{Def:LStableGeneral}.}
\end{quote}
\begin{proof}
A tedious but straightforward computation shows that 
\begin{equation}\label{Eq:RicciScalar}
 \ricd (N,N) + |A|^2 = \frac{1}{2}\scad - K + {\rm div}(X) + F, 
\end{equation}where $K$ is the Gaussian curvature of $\Sigma$, ${\rm div}$ is the divergence operator on $\Sigma$, and 
\begin{equation*}
\begin{split}
 X &= \nabla  \phi ,\\
 F & = \frac{1}{2}\left( H_\phi ^2 + |A|^2 + |\nabla  \phi|^2\right) .
\end{split}
\end{equation*}

We can also easily check that 
$$ \frac{1}{2}\scad - K + {\rm div}(X) + F = 3 V - K +{\rm div}(X)+ \frac{3}{8}|\nabla  \phi|^2 .$$

Let $f \in C_0 ^\infty (\Sigma) $ and set $u := e ^{-\phi/2} f \in C^\infty _0 (\Sigma)$. Plugging $u= f e ^{-\phi/2}$ in \eqref{Eq:2Variation1} and using the above equalities, we have 
\begin{equation*}
\int_ \Sigma \left( 3 V - K + {\rm div}(X) + \frac{3}{8}|\nabla  \phi |^2\right) f^2dv_\Sigma \leq \int _\Sigma \left( |\nabla f|^2 - f \metag{\nabla f}{\nabla  \phi } + \frac{f^2}{4}|\nabla  \phi|^2\right) dv_\Sigma .
\end{equation*}

Now, the Divergence Theorem yields
$$ \int _\Sigma f^2 {\rm div}(X)  \, dv_\Sigma = -2 \int _\Sigma f \metag{\nabla f}{\nabla  \phi} \, dv_\Sigma ,$$
and using the inequality 
$$  f \metag{\nabla f}{\nabla  \phi } \leq 2 |\nabla f|^2 + \frac{1}{8}|\nabla  \phi|^2 f^2 $$we obtain

\begin{equation*}
\int_ \Sigma \left( 3V - K +\frac{3}{8}|\nabla  \phi|^2 \right ) f^2dv_\Sigma \leq  \int _\Sigma  \left(3 |\nabla f|^2 + \frac{3}{8}|\nabla  \phi |^2 f^2 \right)dv_\Sigma ,
\end{equation*}or equivalently, 
\begin{equation*}
\int_ \Sigma( V - a K) f^2dv_\Sigma \leq  \int _\Sigma  |\nabla f|^2  dv_\Sigma ,
\end{equation*}as desired.
\end{proof}

The above Main Lemma allows us to use the strong machinery developed for Schr\"{o}dinger operators since the pioneer work of T. Colding and W. Minicozzi \cite{TColWMin02} and improved in \cite{PCas06,JEsp12a, JEsp12b,JEspHRos11,WMeeJPerARos08} and references therein. 

The first thing we shall observe is when we can have a complete noncompact weighted $H_\phi -$stable surface under conditions on the P-scalar curvature. The following result is the extension of the diameter estimate for stable $H-$surfaces given by H. Rosenberg \cite{HRos06}:

\begin{theorem}\label{Th:Diameter}
Let $\Sigma \subset (\amb , g , \phi )$ be a weighted $H_\phi -$stable surface with boundary $\partial \Sigma$. If 
$$\scad + \frac{1}{4}|\nabla  \phi|^2+ H^2_\phi \geq c >0 \text{ on } \Sigma,$$ then 
$$ d_\Sigma \left( p , \partial \Sigma \right) \leq \frac{2\pi }{\sqrt{3c}} \text{ for all } p \in \Sigma ,$$
where $d_\Sigma$ denotes the intrinsic distance in $\Sigma$. Moreover, if $\Sigma$ is complete without boundary, then it must be topologically a sphere.
\end{theorem}
\begin{proof}
The condition $\scad + \frac{1}{4}|\nabla  \phi|^2+ H^2_\phi  \geq c >0$ implies that $V \geq c >0$ for some positive constant $c$. Therefore, \cite[Theorem 2.8]{WMeeJPerARos08} with $a=1/3$ implies the result. 
\end{proof}

The above result says that, in a manifold with density $(\amb , g , \phi) $ and $\scad + \frac{1}{4}|\nabla  \phi|^2\geq 0$, the only complete noncompact weighted $H_\phi -$stable surface we shall consider are the $\phi -$minimal ones. So, the next step is to extend the well-know result of D. Fischer-Colbrie and R. Schoen \cite{DFisRSch80} on the topology and conformal type of complete noncompact stable minimal surfaces. P. T. Ho \cite{PHo10} extended Fischer-Colbrie-Schoen result under the additional hypothesis that $|\nabla \phi| $ is bounded and $\scad \geq 0$ on $\Sigma$. 

Here, we drop that condition and relax the hypothesis on the P-scalar curvature. Specifically,

\begin{theorem}\label{Th:FCS}
Let $\Sigma \subset (\amb ^3, g , \phi )$ be a complete (noncompact) weighted stable minimal surface where $\scad + \frac{1}{4}|\nabla  \phi|^2 \geq 0$. Then, $\Sigma$ is conformally equivalent either to the complex plane $\c$ or to the cylinder $\c \setminus \set{0}$. In the latter case, $\Sigma$ is totally geodesic, flat and $\scad  + \frac{1}{4}|\nabla  \phi|^2\equiv 0$ along $\Sigma$.
\end{theorem}
\begin{proof}
From the Main Lemma, $L := \Delta - a K + V $ is stable in the sense of operators with $1/4 < a= 1/3$ and $V \geq 0$. From \cite{PCas06} or \cite{WMeeJPerARos08}, we can see that $\Sigma$ is conformally equivalent either to the plane or to the cylinder. 

In the latter case, either \cite[Lemma 2.1]{JEsp12a} or \cite[Theorem 1.3]{PBerPCas12b} give us that $K \equiv 0$ and $V \equiv 0$, which implies that $|A|^2 \equiv 0$ and $\scad  + \frac{1}{4}|\nabla  \phi|^2\equiv 0$ is constant along $\Sigma$.
\end{proof}

We must point out that the results in this section are nontrivial and do not follow from a conformal change of the ambient metric. 

\section{Applications}\label{Sect:Applications}

\subsection{Mean Curvature Flow}\label{Sect:MCF}

Self-similiar solutions to the mean curvature flow in $\r ^{n+1}$ can be seen as weighted minimal hypersurfaces in the Euclidean space endowed with the corresponding density.  Huisken \cite{GHui90} and T. Colding and W. Minicozzi \cite{TColWMin12a,TColWMin12b} proved this relationship for self-shrinker (and also self-expander). Moreover, T. Illmanen \cite{TIll} showed that translating solitons of the Mean Curvature Flow can be seen as weighted minimal hypersurfaces. We will explain this in more detail. 

\subsubsection{Self-Similar solutions}

Let $X: (0,T) \times \Sigma \to \r ^{n+1}$ be a one parameter family of smooth hypersurfaces moving by its mean curvature, that is, $X$ satisfies
$$ \derv{X}{t} = - H \, N  $$where $N$ is the unit normal along $\Sigma _t = X(t, \Sigma)$ and $H$ is its mean curvature. Self-similiar solutions to the mean curvature flow are a special class of solutions, they correspond to solutions that a later time slice is scaled (up or down depending if it is either expander or shrinker) copy of an early slice. In terms of the mean curvature, $\Sigma$ is said to be a self-similar solution if it satisfies the following equation
\begin{equation}\label{Eq:Selfsimilar}
H = -\frac{c}{2}\meta{x}{N},
\end{equation}where $c=\pm 1$, $x$ is the position vector in $\r ^{n+1}$ and $\meta{\cdot }{ \cdot }$ is the standard Euclidean metric. Here, if $c= -1$ then $\Sigma$ is said a {\bf self-shrinker} and if $c=+1$ then $\Sigma $ is called {\bf self-expander}. 

It is straightforward to check that self-shrinker (resp. self-expander) are weighted minimal hypersurfaces in $(\r ^{n+1}, g_0 , \phi _{-1})$ with density $\phi _{-1} := -\frac{|x|^2}{4}$ (resp. $\phi _{+1} := \frac{|x|^2}{4}$).  Hence, one can define: 

\begin{definition}\label{Def:SelfStable}
We say that a self-shrinker (resp. expander)  $\Sigma$ is {\bf $L-$stable} if it is stable as a weighted minimal hypersuface in $(\r ^{n+1} , \meta{\cdot }{\cdot } , \phi _{-1})$ (resp. $(\r ^{n+1} , \meta{\cdot }{\cdot } , \phi _{+1})$).
\end{definition}

We also do the following definition for the sake of clearness:

\begin{definition}\label{Def:SelfFinite}
Let $\Sigma \subset \r ^{n+1}$ be a complete self-shrinker (resp. expander). We say that $\Sigma $ is {\bf weighted parabolic} if $(\Sigma , g , \phi _{-1})$ is $\phi _{-1}-$parabolic (resp. if $(\Sigma , g ,\phi _{+1})$ is $\phi_{+1}-$parabolic). 
\end{definition}

We will apply Theorem A for extending a recent result of Colding-Minicozzi \cite{TColWMin12a,TColWMin12b} on the nonexistence of certain self-shrinkers. Specifically, 

\begin{theorem}\label{Th:FiniteMCF}
There are no complete weighted parabolic $L-$stable self-shrinkers in $\r ^{n+1}$.
\end{theorem}
\begin{proof}
Let $\Sigma \subset (\r^{n+1},\meta{\cdot }{\cdot } , \phi _{-1})$ be a self-shrinker, i.e., it satisfies $H = \frac{1}{2}\meta{x}{N}$. Assuming $\Sigma$ is $L-$stable, we obtain that the gradient Schr\"{o}dinger operator (see Lemma \ref{Lem:SecondVariation}) given by
$$ L := \Delta  +\metag{\nabla \phi _{-1}}{ \nabla \cdot } + (|A|^2 + \frac{1}{2}) ,$$is stable in the sense of Definition \ref{Def:LStableGeneral}. From Lemma \ref{Lem:Characterization}, there exists a positive subsolution, that is, there exists $u >0$ so that 
$$ L u \leq 0 . $$

Moreover, $q := |A|^2 + \frac{1}{2} \geq \frac{1}{2} >0$. Thus, Theorem \ref{Th:NaberYau} says that $u$ is constant and $q \equiv 0$, which is a contradiction. 
\end{proof}

T. Colding and W. Minicozzi \cite{TColWMin12a,TColWMin12b} proved that there are no $L-$stable self-shrinker $\Sigma$ with polynomial area growth. Polynomial area growth implies that 
$$ \int _{B(i+1) \setminus B(i)} \phi _{-1} \, dv \to 0 ,$$where $B(r)$ denotes the geodesic ball of radius $r$ centered at a fixed point. In particular, $\Sigma$ is weighted parabolic. 

As far as we know, there are no general classification results for $L-$stable self-expanders. We can, at least, give some condition for the nonexistence:

\begin{proposition}\label{Prop:SelfExpander}
There are no complete weighted parabolic $L-$stable self-expanders in $\r ^{n+1}$ with $|A|^2 \geq 1/2$.
\end{proposition}

In fact, there is another notion of stability more related to graphs. Let $\Sigma \subset \r^{n+1}$ be a self-shrinker (expander). Let $a \in \r ^{n+1}$ be a constant vector and consider $u_a := \meta{N}{a}$, then $u_a$ satisfies (see \cite{TColWMin12a,TColWMin12b}):

\begin{equation}\label{Eq:L0}
L_0 u_a := \Delta u_a  + \metag{\nabla \phi _c}{ \nabla u_a} + |A|^2 u_a =0 .
\end{equation} 

If $\Sigma$ is a complete multi-graph, say with respect to the $e_{n+1}$ direction, then $\nu := \meta{N}{e_{n+1}}$ is positive and satisfies \eqref{Eq:L0}, therefore Lemma \ref{Lem:Characterization} yields that $L_0$ is stable in the sense of Definition \ref{Def:LStableGeneral}. This motivates the following:

\begin{definition}\label{Def:L0Stable}
We say that a self-similiar solution to the mean curvature flow $\Sigma$ is {\bf $L_0-$stable} if the gradient Schr\"{o}dinger operator
$$ L_0 := \Delta  +\metag{\nabla \phi _{c}}{ \nabla \cdot } + |A|^2 ,$$is stable in the sense of Definition \ref{Def:LStableGeneral}.
\end{definition}

Therefore, the above discussion leads us to:

\begin{theorem}\label{Th:GraphFlow}
The only complete weighted parabolic self-similar solutions (shrinker or expander) to the mean curvature flow that are $L_0-$stable are hyperplanes.

In particular, The only self-similar solutions (shrinker or expander) to the mean curvature flow that are complete multi-graphs and weighted parabolic are hyperplanes. Moreover, the only self-shrinkers that are entire graphs are the hyperplanes.
\end{theorem}

\begin{remark}
It would be interesting to know is a self-shrinker that is a complete multi-graph must be an entire graph.
\end{remark}

\subsubsection{Translating Soliton}

A translating soliton (cf. \cite{TIll}) for the Mean Curvature are nothing but complete surfaces in $\r ^{n+1}$ whose evolution  under the flow is invariant by a translations. Analytically, they can be described as hypersurfaces $\Sigma \subset \r^{n+1}$ satisfying 
$$ H + \meta{N}{v} = 0 ,$$where $v \in \s ^{n}$ is a fixed direction. Hence, we can see that a translating soliton can be see as a weighted minimal hypersurface 
$$ \Sigma \subset (\r ^{n+1} ,\meta{}{}, \phi _T) \text{ where  }  \phi _T (x) = \meta{x}{v}.$$

Analogously as we did above, we say that $\Sigma  \subset (\r ^{n+1} ,\meta{}{}, \phi _T) $ is {\bf weighted parabolic} if $(\Sigma , g , \phi _T)$ is $\phi _T-$parabolic. Here $g$ is the induced metric.  

For simplicity, we assume that $v = e_{n+1}$. It is clear from the definition of translating soliton that any vertical hyperplane $P := \set{ x\in \r ^{n+1} \, :\,\, \meta{x}{e_{n+1} } = 0}$ is a translating soliton. 

\begin{proposition}\label{Prop:Plane}
A vertical hyperplane is not weighted parabolic. 
\end{proposition}
\begin{proof}
Up to change of coordinates, we can assume $P = \set{x \in \r ^{n+1} \, : \,\, x_1 = 0}$. Consider $u(x_2, \ldots , x_{n+1}) = e^{-x_{n+1}} $ which is positive on $P$. Since $\phi _ T (x) = \meta{x}{e_{n+1}} $ we obtain $\camb \phi _ T = e_{n+1}$. Hence
$$ \nabla u + \meta{e_{n+1}}{\nabla u} = 0 ,$$that is, $u$ is a $\phi _T -$harmonic positive function on $P$. Thus, $P$ is not weighted parabolic.
\end{proof}

Note that $\ricd \equiv 0$ since $\camb \phi _T = v$ and the Ricci curvature of the Euclidean space vanishes identically. Hence, we obtain: 

\begin{theorem}\label{Th:Translating}
There are no complete weighted parabolic stable translating soliton.
\end{theorem}
\begin{proof}
Suppose there exists $\Sigma $ a complete weighted parabolic stable translating soliton. Then, $\Sigma \subset (\r ^{n+1},\meta{}{}, \phi _T) $ is a $\phi _T -$parabolic complete stable $\phi _T -$minimal hyeprsurface. As we saw above, $\ricd \equiv 0$, and hence $|A|^2 \equiv 0$ by Theorem \ref{Th:Hypersurface}. Therefore, $\Sigma$ is a hyperplane, which contradicts Proposition \ref{Prop:Plane}.
\end{proof}

\subsection{Gradient Ricci solitons}\label{Sect:RicciFlow}

In G. Perelman's work \cite{GPer} on the Poincar\'{e} conjecture, he was able to formulate the Ricci flow as a gradient flow (we follow Topping\rq{}s Book \cite{PTop06}). Let us consider a (closed) manifold with density $(\amb , g , \phi)$ and introduce the following {\it Fischer information} functional
$$ \mathcal F (g,\phi ) := \int _\amb \left( R + \abs{\camb  \phi }^2 \right) \,dm ,$$
where $dm = e^\phi \, dv$ and $R$ is the scalar curvature of $(\amb, g)$. Equivalently, integrating by parts,
$$ \mathcal F (g ,\phi) := \int _\amb \scad \, dm , $$here,  $\scad$ is the {\bf Perelman Scalar Curvature}, in short P-scalar curvature, given by 
\begin{equation}\label{Eq:Scad}
\scad = R - 2 \Delta _g  \phi - \abs{\camb \phi}^2 ,
\end{equation}where $\Delta _g$ and $\camb$ are the Laplacian and Gradient operators with respect to the ambient metric $g$ respectively. 

\begin{remark}
We focus on the compact case, i.e., $\amb$ is closed. Nevertheless, the discussion can be extended to complete manifolds under appropiatte conditions. Moreover, we will consider here the smooth case.
\end{remark}

Let us consider the variation of $\mathcal F$ which preserves the measure $dm :=e^\phi \,dv$. So, the evolution of $\mathcal F$ is given by
$$ \derv{}{t}{} \mathcal F (g,\phi) = -\int _\amb \metag{{\rm Ric} - \camb ^2  \phi}{\parc{g}{t}} \, e^\phi \, dv = -\int _\amb \metag{\ricd }{\parc{g}{t}} e^\phi \, dv .$$

Therefore, if we have a solution to the coupled system
\begin{equation}\label{Eq:Coupled}
\left\{\begin{matrix}
\parc{g}{t} &= & -2 \ricd ,\\[3mm]
\parc{\phi}{t} &=& \left(  R - \Delta _g  \phi \right) \phi  ,
\end{matrix}\right.
\end{equation}then $\mathcal F$ evolves as
$$ \derv{}{t} \mathcal F (g,\phi) = 2 \int _\amb |\ricd |^2 e^\phi \, dv \geq 0 .$$

Summarizing, if we define $dm := e^\phi \, dv$, which is constant in time, we could view $g$ as a gradient flow for the functional $g \to \mathcal F (g,\phi)$, where $\phi$ is determined by $dm$ and $dv$.

Perelman\rq{}s trick was to show that a solution of the coupled system \eqref{Eq:Coupled} is somehow equivalent to a solution to the decoupled system
\begin{equation}\label{Eq:Decoupled}
\left\{\begin{matrix}
\parc{g}{t} &= & -2 {\rm Ric} ,\\[3mm]
\parc{\phi}{t} &=& \left(\scad + \Delta _g  \phi\right) \phi  ,
\end{matrix}\right.
\end{equation}

One way to see this is that solutions to the coupled system \eqref{Eq:Coupled} may be generated by pulling back solutions of the decoupled system \eqref{Eq:Decoupled} by an appropiatted time-dependent diffeomorphism. We can also reverse the argument.

The aim of Perelman for introducing $\mathcal F$ (and the more general $\mathcal W -$entropy) was the classification of (gradient) Ricci solitons, that is, self-similar solutions to the Ricci flow. They can be analytically described as manifolds with density $(\amb , g , \phi)$ such that $\ricd = \lambda g$, with $\lambda \in \r$. They are called {\bf shrinking}, {\bf steady} or {\bf expanding} depending if $\lambda > 0$, $\lambda =0$ or $\lambda <0$ respectively. Steady Ricci solitons appear as critical points of the $\mathcal F$ functional (shrinking Ricci solitons appear as critical points of the $\mathcal W-$entropy functional). See \cite{PTop06} for a detailed exposition.

\begin{remark}
Moreover, Perelman\rq{}s entropy functional and manifolds with density have applications to String Theory (see \cite{HCao09}).
\end{remark}

The condition $\ricd = \lambda g$ says
$$ {\rm Ric} - \camb ^2 \phi = \lambda g ,$$therefore, gradient Ricci solitons can be thought as a natural generalization of Einstein metrics. One important consequence of F. Morgan result \cite{FMor05} is that gradient shrinking solitons are $\phi -$parabolic, which is the key condition to apply Theorem \ref{Th:NaberYau}.

\begin{lemma}\label{Lem:Morgan}
Let $(\Sigma , g , \phi)$ be a complete manifold with density so that $\ricd \geq \lambda$, $\lambda >0$. Then 
$$ \int _\amb e^\phi \, dv < +\infty .$$

In particular, $(\Sigma , g , \phi) $ is $\phi -$parabolic. 
\end{lemma}

Here, we extend a recent result of P. Petersen and W. Wylie \cite{PPetWWyl10} relaxing the hypothesis to a $L^2-$type bound of the scalar curvature $R$. Moreover, they assume that the Weyl tensor vanishes identically, we change that condition by an inequality between the scalar curvature and the Ricci curvature.

\begin{theorem}\label{Th:Shrinker}
Let $(\Sigma , g , \phi , \lambda)$ be a complete shrinking soliton of dimension $n\geq 3$ so that $\lambda R \leq \abs{{\rm Ric} }^2$. Assume that there exists a sequence of cut-off functions $\set{\psi _i}_i \subset C^\infty _0 (\Sigma)$ such that $0 \leq \psi _i \leq 1 $ in $\Sigma$, the compact sets $\Omega _i := \psi _i ^{-1}(1)$ form an increasing exhaustion of $\Sigma$, and
the sequence of weighted energies $\set{ \int _\Sigma |\nabla \psi _i|^2 R ^2 e^\phi \, dv_\Sigma }_i$ is bounded.

Then, $(\Sigma ,g)$ has constant scalar curvature and $\lambda R = \abs{{\rm Ric}}^2 $.
\end{theorem}
\begin{proof}
The proof follows from the following equation developed in \cite{PPetWWyl10}:
\begin{equation}\label{Eq:ScalarLambda}
\Delta R + g(\nabla \phi , \nabla R)- 2\lambda R = - 2\left|{\rm Ric} \right|^2 .
\end{equation}

First, since the scalar curvature satisfies the boundedness condition, $R>0$ or $(\Sigma , g)$ is flat (see \cite{PPetWWyl12}). So, we can continue assuming $R >0$. Second, since $\lambda R -\abs{{\rm Ric}}^2 \leq 0$, \eqref{Eq:ScalarLambda} implies that 
$$  \Delta R + g(\nabla \phi , \nabla R) = 2 \lambda R - 2 \abs{{\rm Ric}}^2 \leq 0 .$$ 

Thus, Theorem \ref{Th:NaberYau} implies that $R$ is constant, and so, $\lambda R = \abs{{\rm Ric}}^2 $.
\end{proof}

We also have the following classification result:

\begin{theorem}\label{Th:GradientSoliton}
\begin{itemize}
\item A complete $\phi-$parabolic gradient expander soliton and nonnegative scalar curvature is Ricci flat. 

\item A complete $\phi-$parabolic gradient steady soliton whose scalar curvature does not change sign is Ricci flat. Moreover, if $\phi$ is not constant then it is a product of a Ricci flat manifold with $\r $. Also, $\Sigma$ is diffeomorphic to $\r ^n$.

\item A complete $\phi -$parabolic gradient shrinking soliton of nonpositive scalar curvature is Ricci flat. Moreover, $(\Sigma , g)$ is isometric to $\r ^n$ endowed with the standard Euclidean metric.
\end{itemize}
\end{theorem}
\begin{proof}
The proof follows from \eqref{Eq:ScalarLambda}.

\begin{itemize}
\item Applying Theorem \ref{Th:NaberYau} to $u=R$ we obtain that $R$ vanishes identically on $\Sigma$, which implies that $(\Sigma , g)$ is Ricci flat from \eqref{Eq:ScalarLambda}. 

\item As above, $R$ satisfies \eqref{Eq:ScalarLambda}. So, taking either $u = -R$ if $R \leq 0$ or $u= R$ if $R\geq 0$, Theorem \ref{Th:NaberYau} implies that $R \equiv 0$. Therefore, $(\Sigma ,g)$ is Ricci flat. The last sentence follows from \cite[Proposition 4]{PPetWWyl12}. Moreover, $\Sigma$ is diffeomorphic to $\r ^n$ by \cite[Proposition 5.7]{HCao09}.

\item Taking  $u = -R$, Theorem \ref{Th:NaberYau} implies that $R \equiv 0$. Therefore, $(\Sigma ,g)$ is Ricci flat. Thus, \cite[Lemma 7.1]{ANab10} implies that $(\Sigma , g)$ is isometric to $\r ^n$ endowed with the standard Euclidean metric.
\end{itemize}
\end{proof}

\subsection{Optimal Transportation}

Recently, geometric problems on manifolds with density $(\amb , g , \phi) $ has been related to corresponding problems in the Wasserstein space of probability measures equipped with the quadratic Wasserstein metric, which corresponds to a relaxed version of Monge\rq{}s optimal transportation problem  (we follow Part II in Villani\rq{}s Book \cite{CVil08}). 

We will explain the relationship between them from a formal point of view, we will not be rigorous (see \cite{CVil08} and references therein). Moreover, we will assume all the objects involved are smooth and verify the appropriate convergence conditions to avoid technical problems and clarify the exposition.

Let $(\amb ,g)$ be a Riemannian manifold. Denote by $dv$ and $d$ the Riemannian measure and distance with respect to $g$ respectively. The Wasserstein space $\mathcal P _2 (\amb)$ is the space of probability measures which have a finite moment of order $2$, i.e., 
$$ \mathcal P _2 (\amb ) = \set{ \mu \in \mathcal P (\amb ) \, : \, \, \int _\amb d(p_0 ,p)^2 \, d\mu < +\infty }, $$where $p_0 \in \amb$ is arbitrary and $\mathcal P (\amb)$ is the space of probability measures. The Wasserstein space will be equipped with the Wasserstein distance
$$ \mathcal W _2 (\mu , \nu) = \left( {\rm inf}\set{  \int_\amb d(p,q)^2 \, d\pi \, : \, \, \pi \in \Pi (\mu , \nu) }\right)^{\frac{1}{2}} ,$$where the infimum is taken over all joint probability measures $\pi $ on $\amb \times \amb $ with marginals $\mu $ and $\nu$. Such joint measures are called {\bf transference plans}; those achieving the infimum are called {\bf optimal transference plans}.

\begin{remark}
We should think on $d^2$ as the cost function on the optimal transportation problem. 
\end{remark} 

F. Otto \cite{FOtt01} noticed that computations of Riemannian nature can shed light on Optimal Transportation Theory. The problem was to establish rules for formally perform differential calculus on $\mathcal P _2 (\amb)$. We focus here on a certain class of functional that appear naturally on the theory, as for example, the Boltzman entropy.

Let us consider $\phi : \amb \to \r $ a smooth function to distort the reference volume measure, i.e., $dm = e^\phi \, dv$, that is, we can consider a manifold with density $(\amb , g, \phi)$. Consider a smooth function $U : \r ^+ \to \r $ which will relate the values of the density of the probability measure and the value of the functional, i.e., 
$$ \mathcal U _m (\mu ) = \int _\amb U(\rho (p)) \, dm \, , \qquad \mu = \rho \, dm . $$

\begin{remark}
We should think on $\mathcal U _m$ as the internal energy of a fluid, i.e., it is like the energy contained on a fluid of density $\rho$. The function $U$ should be thought as a property of the fluid and it might mean some microscopic interaction. It is natural to assume $U(0)=0$.
\end{remark}

In analogy with thermodynamics, we can introduce the {\bf pressure} as:
$$ {\bf p} (\rho ) := \rho \, U\rq{} (\rho) -U(\rho) .$$

\begin{remark}
The physical meaning of the pressure says that if a fluid is enclosed in a domain $\Omega$, then the pressure felt by the boundary $\partial \Omega$ at a point is normal and proportional to $\mathcal {\bf p}$ at that point.
\end{remark}

So, one can consider the total pressure 
$$\int _\amb {\bf p}(\rho) \, dm $$and again, compute the variation of this functional with respect to small variations of the measure, which leads to the {\it iterated pressure}:
$$ {\bf p}_2 (\rho) := \rho \, {\bf p}\rq{}(\rho) - {\bf p}(\rho) .$$

One can see that the pressure and iterated pressure appear when we differentiate the energy functional $\mathcal U _m$ to first and second order respectively. F. Otto \cite{FOtt01} gave an {\it explicit} expression for the gradient and Hessian of the functional $\mathcal U _m$. For a given measure $\mu$, the gradient of $\mathcal U _m$ at $\mu$ is a {\it tangent vector} at $\mu$ in the Wasserstein space, and it is given by
$$ {\rm grad}_\mu \mathcal U _m = -\left( \Delta {\bf p}(\rho) + g(\nabla \phi ,\nabla {\bf p}(\rho)) \right) \, dm ,$$where, remember, $dm = e^\phi \, dv$. The Hessian of $\mathcal U _m$ at $\mu$ is a quadratic form on the tangent space $T_ \mu \mathcal P_2 (\amb)$, but its expression is rather complicated in general. We continue this discussion only for the Boltzman entropy functional, which is given by 
$$ \mathcal H _m (\mu) := \int _\amb \rho \, \ln \rho  \, dm  , \qquad \mu = \rho \, dm .$$

For the Boltzman entropy, its gradient is given by 
$$ {\rm grad}_\mu \mathcal H _m =  -\left( \Delta \rho + g(\nabla \phi  ,\nabla \rho) \right) \, dm $$so, we can see that critical points of the Boltzman entropy correspond to positive solution to 
$$  \Delta \rho + g(\nabla \phi ,\nabla \rho)  =0 ,$$or, equivalently, they correspond to positive {\it Jacobi} functions for the gradient Schr\"{o}dinger operator
$$ L _\phi \rho :=  \Delta \rho + g(\nabla \phi ,\nabla \rho) .$$

Now, since optimal transportation plans (for the Boltzman entropy functional) must be critical points, we call them {\bf critical transportation plans}, we can locate the space of optimal transportation plans under global conditions on the initial data $(\amb , g , \phi)$ using Theorem A, that is:

\begin{theorem}
Let $(\amb , g , \phi)$ be a $\phi -$parabolic complete manifolds. Then, critical transportation plans $\mu$ for the Boltzman entropy functional are $\mu := \alpha \,dm $, $\alpha \in \r ^+$. Moreover, if we assume 
$$ \int _\amb d\mu = \int _\amb dm ,$$then, the only optimal transportation plan is $\mu := e^\phi \, dv$.
\end{theorem}

Actually, this is the situation for the Gaussian measure space $(\r ^{n+1}, \meta{\cdot }{\cdot } , e^{-\pi |x|^2})$. Note that, if $dm := e^\phi \, dv \in \mathcal P _2 (\amb )$, then $(\amb , g , \phi)$ has finite $\phi-$capacity.

\begin{remark}
It would be interesting to investigate deeper this relationship, even for other functionals and not only for the Boltzman entropy.
\end{remark}

\begin{center}
{\bf Acknowledgements}
\end{center}
\begin{quote}
The author wishes to express his gratitude to F. Morgan (Williams College, USA) and F. Marques (IMPA, Brazil) for their comments and corrections. Also, the author wishes to express his gratitude to T. Bourni (Frei University, Germany) and L. Hauswirth (Universit\'{e} Paris-Est, France) for pointing us out the relation between manifolds with density and Mean Curvature Flow and Optimal Transport respectively. Also, we appreciate the comments of L. Mari (UFC, Brazil) and E. Barbosa (UFMG, Brazil).

The author is partially supported by Spanish MEC-FEDER Grant MTM2013-43970-P; CNPq-Brazil Grants 405732/2013-9 and 14/2012 - Universal, Grant 302669/2011-6 - Produtividade;  FAPERJ Grant 25/2014 - Jovem Cientista de Nosso Estado.
\end{quote}

\end{document}